\documentclass[a4paper,12pt]{amsart}
\usepackage{tabularx}
\usepackage[dvips]{graphicx}
\usepackage{amsmath}
\usepackage{amssymb}
\usepackage{amsbsy}

\usepackage{amsgen}
\usepackage{amsopn}
\usepackage{amscd}
\usepackage{amstext}
\usepackage{amsthm}

\theoremstyle{definition}
\newtheorem{definition}{Definition}[section]
\theoremstyle{plain}
\newtheorem{theorem}{Theorem}[section]

\newtheorem{lemma}{Lemma}[section]
\newtheorem{corollary}{Corollary}[section]
\newtheorem{proposition}{Proposition}[section]

\theoremstyle{remark}
\newtheorem{remark}{Remark}[section]

\theoremstyle{definition}

\newtheorem{acknow}{Acknowledgment}

\begin{document}
\title{On Generalization of Homotopy of Words and Its Applications}

\author{FUKUNAGA Tomonori}\thanks{The author is JSPS research fellow (DC).
This work was supported by KAKENHI}

\maketitle

\begin{abstract}
V. Turaev introduced the theory of topology of words and phrases 
in 2005. This is a combinatorialy extension of the theory of 
virtual knots and links. In this paper we generalize the 
notion of homotopy of words and phrases and 
we give geometric meanings of the generalized homotopy of words. 
Moreover using the generalized homotopy theory of words and phrases,
we extend some homotopy invariants of nanophrases to 
$S$-homotopy invariant of nanowords with some homotopy data $S$.                
\end{abstract}

{\bf keywords:} nanowords, nanophrases, homotopy of words and phrases, 
curves on surfaces, ornaments. 

\markboth{FUKUNAGA Tomonori}{On Generalization of Homotopy of Words and Its Applications}

\section{Introduction.}
In the papers \cite{tu1} and \cite{tu2}, 
V. Turaev introduced the theory of topology of words and phrases.
This is a combinatorialy extension of the theory of virtual 
strings or the virtual knots and links 
(see \cite{kad}, \cite{ka}, \cite{sw} and \cite{tu4} for example).\par
In this paper, words are finite sequences of letters
in a given alphabet, letters are elements of an alphabet and phrases
are finite sequences of words. V. Turaev introduced generalized words and 
phrases in \cite{tu1} and \cite{tu2} which are called \'etale words
and \'etale phrases. Let $\alpha$ be an alphabet
endowed with an involution $\tau:\alpha \rightarrow \alpha$. Let
$\mathcal{A}$ be an alphabet endowed with a mapping 
$|\cdot|:\mathcal{A} \rightarrow \alpha$ which is called a projection.
We call this $\mathcal{A}$ an $\alpha$-alphabet. Then we call a
pair an $\alpha$-alphabet $\mathcal{A}$ and a word on $\mathcal{A}$ 
an \'etale word. 
If all letters in $\mathcal{A}$ appear exactly twice, then we call
this \'etale word a nanoword. Similarly we call a pair
an $\alpha$-alphabet $\mathcal{A}$ and a phrase on $\mathcal{A}$ 
an \'etale phrase. Further if each letters in $\mathcal{A}$ 
appear exactly twice, then we call this \'etale phrase a nanophrase.\par
Turaev introduced an equivalence relation which is called $S$-homotopy on 
the set of  
nanowords and nanophrases for a fixed subset $S$ of $\alpha$ $\times$ $\alpha$
$\times$ $\alpha$ which is called homotopy data. 
$S$-homotopy of nanowords and nanophrases 
is generated by isomorphism, and three homotopy moves. 
The first homotopy move is deformation
that changes $xAAy$ into $xy$. The second homotopy move is deformation that
changes $xAByBAz$ into $xyz$ when $|A|$ is equal to $\tau(|B|)$. 
The third homotopy move is deformation 
that changes $xAByACzBCt$ into $xBAyCAzCBt$ when a triple $(|A|,|B|,|C|)$
is a element of homotopy data $S$. This equivalence relation is suggested 
by the Reidemeister moves in the theory of knots. Moreover in the paper 
\cite{tu2},
Turaev showed special cases of the theory of topology of phrases 
corresponds to the theory of 
stable equivalent classes of ordered, pointed, oriented multi-component
curves on surfaces and knot diagrams. Note that the theory of
stable equivalence classes of ordered, pointed, oriented multi-component
curves on surfaces (respectively link diagrams) is equivalent to the theory
of ordered open flat virtual links (respectively ordered open virtual
links).\par
Now in this paper, we generalize the notion of the $S$-homotopy
moves in Turaev's theory of words. For fixed sets $Q \subset \alpha$ and
$R \subset \alpha \times \alpha$, we replace the first homotopy move
and the second homotopy moves in Turaev's theory of words as follows.
The first homotopy move is deformation that changes $xAAy$ into $xy$
when $|A|$ is an element of $Q$ and  the second homotopy move is 
deformation that changes $xAByBAz$ into $xyz$ 
when a pair $|A|$ and $|B|$ is an element of $R$.
We call an equivalence relation which is generated by 
isomorphism, and three new homotopy moves $(Q,R,S)$-homotopy of 
nanophrases. When $Q$ is equal to $\alpha$ and $R$ is graph of 
$\tau$, we obtain Turaev's $S$-homotopy of nanophrases. 
Moreover we give the geometric meanings of the generalized homotopy
theory of words and phrases. We show the generalized homotopy 
theory of words in common generalization of the theory of the 
multi-component virtual strings and the theory of the ornament.
On the other hand, we construct some $(Q,R,S)$-homotopy invariants for 
nanowords from $S$-homotopy invariants in Turaev's theory of words.
In the paper \cite{gi1}, A.Gibson constructed a homotopy invariant for
Gauss phrases (in other words, nanophrases over the one-element set)
which is called $S_{o}$. Moreover Gibson showed the invariant $S_{o}$
is strictly stronger than the invariant $T$ for Gauss phrases
which was defined by the author in \cite{fu1}.  
In this paper we extend Gibson's $S_{o}$ invariant for Gauss phrases to 
$(\Delta_{\alpha^{3}})_{\alpha,k}$-homotopy invariant for nanowords over 
$\alpha_{k}$ where $\Delta_{\alpha^{3}}$ is the diagonal set of 
$\alpha^{3}$ and
$\alpha_{k}$ and $S_{\alpha,k}$ are defined in Section \ref{gmqrs}.
To do this, we extend Gibson's $S_{o}$ invariant to 
$(\Delta_{\alpha^{3}})$-homotopy invariant for nanophrases over 
any $\alpha$ 
(This extension problem was mentioned in the paper \cite{gi1}).\par
The rest of this paper is organized as follows.
In Section \ref{s2}, 
we review the theory of topology of nanowords and nanophrases.  
In Section 3 we introduce the generalized homotopy of nanowords and 
nanophrases. 
In Section \ref{gm}, we discuss the geometric meanings of 
$S$-homotopy of nanophrases and the generalized homotopy 
of nanowords. 
In Section \ref{ext}, we construct $(Q,R,S)$-homotopy invariants
from some $S$-homotopy invariants of nanophrases which was defined in
papers \cite{fu1}, \cite{fu2} and \cite{gi1}. Moreover we show some of
them are $S$-homotopy invariant for nanowords with 
some homotopy data $S$ which is grater than the diagonal set of 
$\alpha^{3}$.   
\section{Review of Turaev's Theory of Words and Phrases}\label{s2}
In this section we review Turaev's theory of words and phrases
which was introduced by V. Turaev in papers \cite{tu1} and \cite{tu2}.
We can find a survey of Turaev's theory of words in the paper
\cite{tu3}.
\subsection{\'Etale words and Nanowords.}
First we prepare some terminologies. In this paper 
an \emph{alphabet} means a finite set and a \emph{letter}
means an element of an alphabet. For an alphabet $\mathcal{A}$ and
$n \in \mathbb{N}$,
a \emph{word} on $\mathcal{A}$ of length $n$ is a map 
$w:\hat{n} \rightarrow \mathcal{A}$ where $\hat{n}$ is 
$\{1,2,\cdots,n\}$. We denote a word of length $n$ by
$w(1)w(2)\cdots w(n)$.   
We regard the map from empty set to empty set as 
the word of length $0$ and denote it by $\emptyset$. 
A \emph{phrase} of length $k$ on $\mathcal{A}$ is 
a sequence of words $w_{1},w_{2}, \cdots , w_{k}$ on $\mathcal{A}$.
We denote this sequence by $(w_{1}|w_{2}| \cdots |w_{k})$.
We call the number $\sum_{i=1}^{k}$(length of $w_{i}$)   
\emph{number of letters} of the phrase.
Especially if each letter in $\mathcal{A}$ appear exactly twice 
in a word $w$ on $\mathcal{A}$, then we call this word $w$ a 
\emph{Gauss word}. Similarly
for a phrase $P$ on $\mathcal{A}$ if each letter in $\mathcal{A}$
appear exactly twice in $P$, then we call $P$ a \emph{Gauss phrase}
(C. F. Gauss studied topology of plane curves using Gauss words. See
\cite{ga} for more details).\par   
In \cite{tu1} and \cite{tu2}, Turaev introduced generalized words and 
phrases. Let $\alpha$ be an alphabet endowed with an involution 
$\tau:\alpha \rightarrow \alpha$. Then an \emph{$\alpha$-alphabet}
is a pair of an alphabet $\mathcal{A}$ and a map 
$| \cdot |:\mathcal{A} \rightarrow \alpha$. We call this map
$| \cdot |$ \emph{projection} and we denote the image of 
a letter $A \in \mathcal{A}$ under the projection $|A|$.
We call $|A|$ \emph{projection of $A$}. Now we define 
generalized words and Gauss words which is called 
\'etale words and nanowords.
An \emph{\'etale word} over $\alpha$ is a pair 
(an $\alpha$-alphabet $\mathcal{A}$, a word $w$ on $\mathcal{A}$).
We call length of $w$ \emph{length of \'etale word} $(\mathcal{A},w)$.
Especially if $w$ is a Gauss word on $\mathcal{A}$, 
then we call $(\mathcal{A},w)$ 
a \emph{nanoword}.
Next we define generalized phrases and Gauss phrases
which is called \'etale phrases and nanophrases.
An \emph{\'etale phrase} over $\alpha$ is a pair
(an $\alpha$-alphabet $\mathcal{A}$, a phrase $P$ on $\mathcal{A}$).
We call length of $P$ (respectively numbers of letters) 
\emph{length of \'etale phrase} (respectively 
numbers of letters) $(\mathcal{A},P)$. 
Especially if $P$ is a Gauss phrase on $\mathcal{A}$, 
then we call $(\mathcal{A},w)$ a \emph{nanophrase}.
\subsection{$S$-homotopy of Nanophrases.}
In the paper \cite{tu1} Turaev defined an equivalence relation
on nanophrases which is called $S$-homotopy. This is suggested
by the Reidemeister moves in the theory on knots. In this subsection,
we introduce $S$-homotopy theory of words and phrases. \par
To define $S$-homotopy of nanophrases we prepare some definitions.
First we define isomorphic of nanophrases.
\begin{definition}
Let $(\mathcal{A}_{1},(w_{1}| \cdots |w_{k}))$ and 
$(\mathcal{A}_{2},(v_{1}| \cdots |v_{k}))$ be
nanophrases of length $k$ over an alphabet $\alpha$.
Then $(\mathcal{A}_{1},(w_{1}| \cdots |w_{k}))$ and 
$(\mathcal{A}_{2},(v_{1}|\cdots|v_{k}))$ are \emph{isomorphic}
if there exist a bijection $\varphi$ between $\mathcal{A}_{1}$ 
and $\mathcal{A}_{2}$ such that $|A| = |\varphi(A)|$  for all 
$A \in \mathcal{A}_{1}$ and 
$v_{j} = \varphi \circ w_{j}$ for each $j \in \hat{k}$.  
\end{definition} 
Next we define $S$-homotopy moves of nanophrases.
\begin{definition}
Let $S$ be a subset of $\alpha \times \alpha \times \alpha$. Then
we define \textit{S-homotopy moves} (1) - (3) of nanophrases as follows:
 \par
 (1) $(\mathcal{A} , (xAAy)) \longrightarrow 
(\mathcal{A} \setminus \{ A \} , (xy))$ 
 \par \enskip \enskip \enskip 
for all $A \in \mathcal{A}$ and $x,y$ are sequences of letters in 
$\mathcal{A} \setminus \{ A \}$, possibly including \par
\enskip \enskip \enskip the $|$ character.
 \par
(2) $(\mathcal{A} , (xAByBAz)) \longrightarrow (\mathcal{A} \setminus \{ A , B \} , (xyz))$
   \par \enskip \enskip \enskip 
if $A , B \in \mathcal{A}$ satisfy $|B| = \tau (|A|)$. $x,y,z$ are sequences of
letters in 
$\mathcal{A} \setminus \{A,B\}$, \par
\enskip \enskip \enskip possibly including $|$ character.
 \par
(3) $(\mathcal{A} , (xAByACzBCt)) \longrightarrow (\mathcal{A} , (xBAyCAzCBt))$
 \par \enskip \enskip \enskip if $A,B,C \in \mathcal{A}$ 
satisfy $(|A|,|B|,|C|) \in S$. $x,y,z,t$ are sequences of letters in \par
\enskip \enskip \enskip $\mathcal{A}$, possibly including $|$ character.
\end{definition}
We call this $S$ \emph{homotopy data}.\par
Now we define \emph{$S$-homotopy} of nanophrases.
\begin{definition}
Let $(\mathcal{A}_{1},P_{1})$ and $(\mathcal{A}_{2},P_{2})$ be 
nanophrases over $\alpha$.
Then $(\mathcal{A}_{1},P_{1})$ and $(\mathcal{A}_{2},P_{2})$ are
\emph{$S$-homotopic} (denote $\simeq_{S}$) 
if they are related by a finite sequence of 
isomorphism, $S$-homotopy moves (1) - (3) and inverse of (1) - (3). 
\end{definition}
\begin{remark}
$S$-homotopy moves and isomorphism of nanophrases do not change length of 
nanophrases. Thus for two different integers $k_{1}$ and $k_{2}$, 
a nanophrase of length $k_{1}$ and 
a nanophrase of length $k_{2}$
are not homotopic each other.
\end{remark}
Especially if $S$ is the diagonal set of $\alpha \times \alpha 
\times \alpha$, then we call $S$-homotopy \emph{homotopy}.\par
We denote the set $\{ \rm{Nanophrases \ of \ length k \ over \ \alpha} 
\} / (S-\rm{homotopy})$ by $\mathcal{P}_{k}(\alpha,S)$ and 
$\mathcal{P}_{1}(\alpha,S)$ by $\mathcal{N}(\alpha,S)$.
We recall two lemmas from \cite{tu1} and \cite{tu2}.
\begin{lemma}
 Let $(\alpha,S)$ be homotopy data and $\mathcal{A}$ be an $\alpha$-alphabet.
 Let $A,B,C$ be distinct letters in $\mathcal{A}$ and let
 $x,y,z,t$ be words in $\mathcal{A} \setminus \{A,B,C\}$ such that $xyzt$ is
 a Gauss word. Then the following (i)-(iii) hold : \par
 (i) $(\mathcal{A} , xAByCAzBCt) \simeq_S (\mathcal{A} , xBAyACzCBt)$
 \par \enskip \enskip \enskip if $(|A|,\tau(|B|),|C|) \in S$, \par
 (ii) $(\mathcal{A} , xAByCAzCBt) \simeq_S (\mathcal{A} , xBAyACzBCt)$
 \par \enskip \enskip \enskip if $(\tau(|A|),\tau(|B|),|C|) \in
 S$, \par
 (iii)$(\mathcal{A} , xAByACzCBt) \simeq_S (\mathcal{A} , xBAyCAzBCt)$
 \par \enskip \enskip \enskip if $(|A|,\tau(|B|),\tau(|C|)) \in S$.
\end{lemma}
\begin{lemma}
 Suppose that $S \cap (\alpha \times b \times b) \neq \emptyset$ for
 all $b \in \alpha$. Let $(\mathcal{A} , xAByABz)$ be a nanoword over
 $\alpha$ with $|B|=\tau(|A|)$ where $x,y,z$ are words in $\mathcal{A}
 \setminus \{A,B\}$ such that $xyz$ is a Gauss word. Then \\
 $$(\mathcal{A},xAByABz) \simeq_S (\mathcal{A} \setminus \{A,B\} , xyz).$$
\end{lemma}
In \cite{tu1}, Turaev constructed some homotopy invariants for
nanowords and gave the classification of nanowords of 
length less than or equal to $6$ up to homotopy. 
In \cite{fu2} the author constructed some homotopy
invariants for nanophrases and gave the homotopy classification of 
nanophrases with less than or equal to four letters without
condition on length of phrases. 
\section{Generalized Homotopy of Nanophrases.}
In this section we generalize the notion of $S$-homotopy.
Geometric meanings of generalized homotopy is discussed in
the next section.\par
First we fix two sets 
$Q \subset \alpha$, $R \subset \alpha \times \alpha$
and $S \subset \alpha \times \alpha \times \alpha$.
Then we define \emph{$(Q,R,S)$-homotopy} of nanophrases as follows.
\begin{definition}
We define \textit{$(Q,R,S)$-homotopy moves} (1) - (3) of
nanophrases as follows:\par
 (1) $(\mathcal{A} , (xAAy)) \longrightarrow 
(\mathcal{A} \setminus \{ A \} , (xy))$ 
 \par \enskip \enskip \enskip 
if $|A| \in Q$ and $x,y$ are sequences of letters in 
$\mathcal{A} \setminus \{ A \}$, possibly including \par
\enskip \enskip \enskip the $|$ character.
 \par
(2) $(\mathcal{A} , (xAByBAz)) \longrightarrow (\mathcal{A} \setminus \{ A , B \} , (xyz))$
   \par \enskip \enskip \enskip 
if $A , B \in \mathcal{A}$ satisfy $(|A|,|B|) \in R$. $x,y,z$ are sequences of
letters in 
$\mathcal{A} \setminus \{A,B\}$, \par
\enskip \enskip \enskip possibly including $|$ character.
 \par
(3) $(\mathcal{A} , (xAByACzBCt)) \longrightarrow (\mathcal{A} , (xBAyCAzCBt))$
 \par \enskip \enskip \enskip if $A,B,C \in \mathcal{A}$ 
satisfy $(|A|,|B|,|C|) \in S$. $x,y,z,t$ are sequences of letters in \par
\enskip \enskip \enskip $\mathcal{A}$, possibly including $|$ character.
\end{definition}
\begin{definition}
Let $(\mathcal{A}_{1},P_{1})$ and $(\mathcal{A}_{2},P_{2})$ be 
nanophrases over $\alpha$.
Then $(\mathcal{A}_{1},P_{1})$ and $(\mathcal{A}_{2},P_{2})$ are
\emph{$(Q,R,S)$-homotopic} (denote $\simeq_{(Q,R,S)}$) 
if they are related by a finite sequence of 
isomorphism, $(Q,R,S)$-homotopy moves (1) - (3) and inverse of (1) - (3). 
\end{definition}
We denote 
the set $\{ \rm{nanophrases \ of \ length \ k \ over \ } \alpha \} / 
((Q,R,S)-\rm{homotopy})$ by
 $\mathcal{P}_{k}(\alpha,Q,R,S)$
and $\mathcal{P}_{1}(\alpha,Q,R,S)$ by 
$\mathcal{N}(\alpha,Q,R,S)$.
Note that if we set $Q = \alpha$ and $R = \{(a,\tau(a))|a \in \mathcal{A} \}$,
then $(Q,R,S)$-homotopy is coincide to $S$-homotopy.
\begin{remark}
Turaev considered generalization of the second homotopy move 
in \cite{tu1} Section 3.6. We can see Turaev's generalization 
is equivalent to $(\alpha,R,S)$-homotopy.  
\end{remark}
In the remaining part of the paper we assume that $R$ is the
graph of $\tau$, that is $R = \{(a,\tau(a))\}_{a \in \alpha}$.
\section{Geometric Meanings of Generalized Homotopy of Nanowords.}\label{gm}
In this section we discuss geometric meanings of $(Q,R,S)$-homotopy
of nanowords.
\subsection{Stable equivalence of curves on surfaces.}
In this subsection we introduce stable equivalence of 
curves on surfaces. 
First we define some terminologies.
Through this paper a \emph{curve} means the image of a 
generic immersion of an oriented
circle into an oriented surface. The word ``generic'' means that the
curve has only a finite set of self-intersections which are all double
and transversal. A \emph{$k$-component curve} is defined in the same
way as a curve with the difference that they may be formed by $k$
curves. These curves are called \emph{components}
of the $k$-component curve. A $k$-component curves are \emph{pointed}
if each component is endowed with a base point (the origin) distinct
from the crossing points of the $k$-component curve. A $k$-component
curve is \emph{ordered} if its components are numerated. 
Next we introduce an equivalence relation which is called stably equivalent.
Two ordered, pointed curves are 
\emph{stably homeomorphic} if there is an orientation
 preserving homeomorphism of their regular neighborhoods in the ambient
surfaces mapping the first multi-component curve onto the second one and
preserving the order, the origins, and the orientations of the
components.\par
Now we define stable equivalence of ordered, pointed multi-component
curves \cite{ka}: Two ordered, pointed multi-component 
curves are \emph{stably equivalent} if they can be related by a finite
sequence of the following transformations: (i) a move
replacing an ordered, pointed multi-component curve with a 
stably homeomorphic 
one; (ii) the flattened Reidemeister moves away
from the origin as in Fig. \ref{fig1}.

\begin{figure}
\centerline{\includegraphics[width=9cm]{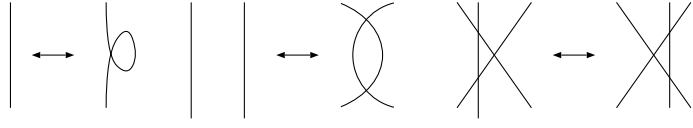}}
\caption{The flattened Reidemeister moves.}\label{fig1}
\end{figure}

\par 
We denote the set of stable equivalence classes of ordered, pointed 
$k$-component curves by $\mathcal{C}_k$.
\begin{remark}
The theory of stable equivalence of curves is closely
related to the theory of virtual strings.
See \cite{sw} and \cite{tu4} for more details.
\end{remark}  
\subsection{Geometric meanings of $S$-homotopy of  nanophrases.}
In the paper \cite{tu2} Turaev gave geometric meanings 
of $S$-homotopy of nanophrases over $\alpha$ with an involution 
$\tau$ for some $\alpha$, $S$ and $\tau$. More precisely, 
Turaev proved a following theorem.
\begin{theorem}\label{thmt1}
There is a canonical bijection between $\mathcal{C}_{k}$
and $\mathcal{P}_{k}(\alpha_{0},S_{0})$ where
$\alpha_{0} = \{a,b\}$ with an involution $\tau_{0}(a)$ $=$ $b$
and $S_{0}=\{(a,a,a),(b,b,b)\}$.  
\end{theorem} 
Moreover let $\mathcal{L}_{k}$ be the set of \emph{stable equivalence 
classes of $k$-component pointed ordered oriented link diagrams} 
(definition of the stable equivalence of link diagrams is given
in \cite{tu2} for example).
Then Turaev proved following theorem.
\begin{theorem}\label{thmt2}
There is a canonical bijection between $\mathcal{L}_{k}$
and $\mathcal{P}_{k}(\alpha_{\ast},S_{\ast})$ where
 $\alpha_{\ast} = \{a_{+},a_{-},b_{+},b_{-}\}$ with
an involution $\tau_{\ast}(a_{\pm})=b_{\mp}$ and  
$S_{\ast}$$=$$\{(a_{\pm},a_{\pm},a_{\pm}),$
$(a_{\pm},a_{\pm},$$a_{\mp}),$
$(a_{\mp},$$a_{\pm},$$a_{\pm}),$
$(b_{\pm},b_{\pm},b_{\pm}),$
$(b_{\pm},b_{\pm},b_{\mp}),$
$(b_{\mp},b_{\pm},b_{\pm})\}$.
\end{theorem}
\subsection{Geometric meanings of generalized homotopy of nanophrases.}
\label{gmqrs}
Now we describe geometric meanings of generalized homotopy
of nanophrases. To do this we prepare some notations.
For an alphabet $\alpha$ and an involution $\tau:\alpha \rightarrow \alpha$
and a set $S \subset \alpha \times \alpha \times \alpha$,
we set 
$\alpha_{k} = \{ a_{ij} | a \in \alpha, i \le j \in \hat{k} \}$, 
$\tau_{k}:\alpha_{k} \rightarrow \alpha_{k}$ 
such that $\tau_{k}(a_{ij}) = \tau(a)_{ij}$ 
for each $a_{ij} \in \alpha_{k}$, 
$Q_{\alpha,k} = \{a_{ii}|a \in \alpha,  i \in \hat{k} \}$,
$R_{\alpha,k} = \{(a_{ij}, \tau(a)_{ij})| a \in \alpha, i \le j \in \hat{k} \}$
and
$S_{\alpha,k} = \{(a_{ij},b_{il},c_{jl})|(a,b,c) \in S, 1 \le i \le j \le 
l \le \hat{k} \}$.
Then we obtain following theorems. 
\begin{theorem}\label{thm1}
Let $\alpha_{0} = \{a,b\}$ and  $S_{0}=\{(a,a,a),(b,b,b)\}$.
There is a canonical injection between 
$\mathcal{C}_{k}$ and 
$\mathcal{N}((\alpha_{0})_{k},Q_{\alpha_{0},k},
R_{\alpha_{0},k},(S_{0})_{\alpha_{0},k})$.
\end{theorem} 
\begin{theorem}\label{thm2}
Let $\alpha_{\ast} = \{a_{+},a_{-},b_{+},b_{-}\}$ and 
$S_{\ast}$$=$$\{(a_{\pm},a_{\pm},a_{\pm}),$
$(a_{\pm},a_{\pm},$$a_{\mp}),$
$(a_{\mp},a_{\pm},$\\
$a_{\pm}),$
$(b_{\pm},b_{\pm},b_{\pm}),$
$(b_{\pm},b_{\pm},b_{\mp}),$
$(b_{\mp},b_{\pm},b_{\pm})\}$.
There is a canonical injection between 
$\mathcal{C}_{k}$ and 
$\mathcal{N}((\alpha_{\ast})_{k},Q_{\alpha_{\ast},k},
R_{\alpha_{\ast},k},(S_{\ast})_{\alpha_{\ast},k})$.
\end{theorem}
To prove these theorems, we prove a more general statement.
\begin{theorem}\label{thm3}
There is a canonical injection between 
$\mathcal{P}_{k}(\alpha,S)$ and 
$\mathcal{N}(\alpha_{k},Q_{\alpha,k},$ $R_{\alpha,k},$ $S_{\alpha,k})$.
\end{theorem}
\begin{proof}
The method of a making nanoword $\varphi(P)$ over $\alpha$ 
from nanophrase $P$ of length $k$ over $\alpha$  
is as follows. Let $P$ be 
$(\mathcal{A},P=(w_{1}|\cdots|w_{k}))$. Then the 
nanoword over $\alpha$ $\varphi(P)$ is defined by
$\varphi(P)$ $=$ $(\mathcal{A}^{\prime},w)$ where the $\alpha$-alphabet
$\mathcal{A}^{\prime}$ is equal to $\mathcal{A}$ as a set,
the projection of $\mathcal{A}^{\prime}$ is defined by 
$|A|_{\mathcal{A}^{\prime}}=(|A|_{\mathcal{A}})_{ij}$ for each 
$A \in \mathcal{A}^{\prime}$ which appears in the $i$-th component
and the $j$-th component in the phrase.
The word $w$ is defined by $w=w_{1}w_{2}\cdots w_{k}$. 
In the remaining of this paper 
we denote $|A|_{\mathcal{A}^{\prime}}$ by $\|A\|$
and $|A|_{\mathcal{A}}$ by $|A|$. 
We put this correspondence $\varphi$.\par
Now we show $\varphi$ induces an injection
$\varphi_{\bullet}:\mathcal{P}_{k}(\alpha,S) \rightarrow  
\mathcal{N}(\alpha_{k},Q_{\alpha,k},R_{\alpha,k},S_{\alpha,k})$.\par
$\bullet$ On well-definedness of $\varphi_{\bullet}$.\par
We show if nanophrases $P_{1}$ and $P_{2}$ are $S$-homotopic,
then $\varphi(P_{1})$ and $\varphi(P_{2})$ are 
$(Q_{k},R_{k},S_{k})$-homotopic.\par
It is clear that if $P_{1}$ and $P_{2}$ are isomorphic,
then $\varphi(P_{1})$ and $\varphi(P_{2})$ are isomorphic.\par
Consider the first homotopy move 
$$P_{1}:=(\mathcal{A}, (xAAy)) \longrightarrow P_{2}:=(\mathcal{A} \setminus
\{A\}, (xy))$$
where $x$ and $y$ are words on $\mathcal{A}$, 
possibly including "$|$" character.\par 
Then $\varphi(P_{1})$ is equal to $xAAy$ and $\varphi(P_{2})$ is 
equal to $xy$. 
In this case two $A$ belong to the same component of $P$.
This implies $\|A\|$ is equal to $|A|_{ii}$ for some $i \in \hat{k}$,
in other words, $\|A\| \in Q_{\alpha,k}$. 
Thus we obtain
$$\varphi(P_{1})=xAAy \simeq_{(Q,R,S)} xy = \varphi(P_{2}).$$
Consider the second homotopy move
$$P_{1}:=(\mathcal{A}, (xAByBAz)) \longrightarrow 
P_{2} := (\mathcal{A} \setminus \{A,B\}, (xyz))$$
where $|A|$ is equal to $\tau(|B|)$, 
and $x$, $y$ and $z$ are words on $\mathcal{A}$
possibly including "$|$" character.\par  
Suppose $A \in \rm{Im} (w_{i}) \cap \rm{Im} (w_{j})$.
Then $B \in \rm{Im}(w_{i}) \cap \rm{Im}(w_{j})$.
Thus $\|A\|$ is equal to $|A|_{ij}$ and $\|B\|$ is equal to
$|B|_{ij}$ by the definition of $\| \cdot \|$. In other words,
$\|A\|$ is equal to $\tau_{k}(\|B\|)$.
This implies 
$$\varphi(P_{1})=xAByBAz \simeq_{(Q,R,S)} xyz = \varphi(P_{2}).$$
Consider the third homotopy move 
$$P_{1}:= (\mathcal{A},(xAByACzBCt)) \rightarrow 
P_{2}:=(\mathcal{A},(xBAyCAzCBt))$$
where $(|A|,|B|,|C|) \in S$, and $x$, $y$, $z$ and $t$ are words on 
$\mathcal{A}$ possibly including "$|$" character.\par
Suppose $A \in \rm{Im}(w_{i}) \cap \rm{Im}(w_{j})$,
$B \in \rm{Im}(w_{i}) \cap \rm{Im}(w_{l})$ and
$C \in \rm{Im}(w_{j}) \cap \rm{Im}(w_{l})$ for 
$1 \le i \le j \le l \le k$. 
Then $\|A\|$ is equal to $|A|_{ij}$, 
$\|B\|$ is equal to $|B|_{il}$ 
and $\|C\|$ is equal to $|C|_{jl}$
by the definition of $\| \cdot \|$.
From the assumption $(|A|,|B|,|C|) \in S$ and $1 \le i \le j \le l \le k$,
we obtain $(\|A\|,\|B\|,\|C\|) \in S_{\alpha,k}$.
Thus 
$$\varphi(P_{1})=xAByACzBCt \simeq_{(Q,R,S)} xBAyCAzCBt = \varphi(P_{2}).$$
By the above $\varphi_{\bullet}$ is well-defined map
between $\mathcal{P}_{k}(\alpha,S)$ and
$\mathcal{N}(\alpha_{k},Q_{\alpha,k},R_{\alpha,k},S_{\alpha,k})$.\par
$\bullet$ On injectibity of $\varphi_{\bullet}.$\par
By the definition of $\varphi$, if $\varphi(P_{1})$ and $\varphi(P_{2})$
are isomorphic, then $P_{1}$ and $P_{2}$ are isomorphic.\par
Consider the first homotopy move 
$$\varphi(P_{1})=(\mathcal{A}^{\prime}, xAAy) \longrightarrow 
\varphi(P_{2})=(\mathcal{A}^{\prime} \setminus \{A\}, xy)$$
if $\|A\| \in Q_{\alpha,k}$ and $x$ and $y$ are words on $\mathcal{A}^{\prime}$, 
Put $\|A\| = x_{ii}$ where $x \in \alpha$ and $i \in \hat{k}$.
Then by the definition of the subscript of $\|A\|$, 
First $A$ and Second $A$ appear same component of $P_{1}$.
This implies $P_{1}$ form   
$(\mathcal{A},(x^{\prime}AAy^{\prime}))$
where $x^{\prime}$ and $y^{\prime}$ are words on $\mathcal{A}$
possibly including "$|$" character.
Thus we obtain
$$P_{1}=(x^{\prime}AAy^{\prime}) \simeq_{S} (x^{\prime}y^{\prime})
=P_{2}.$$
Consider the second homotopy move
$$\varphi(P_{1})=(\mathcal{A}^{\prime}, xAByBAz) \longrightarrow 
\varphi(P_{2})=(\mathcal{A}^{\prime} \setminus \{A,B\}, xyz)$$
where $\|A\|$ is equal to $\tau_{k}(\|B\|)$, 
and $x$, $y$ and $z$ are words on 
$\mathcal{A}^{\prime}$.
Put $\|A\| = |A|_{ij}$ for some $i,j \in \hat{k}$. 
Then $\|B\|$ is equal to $|B|_{ij}$.
Thus $P_{1}$ form $(\mathcal{A}, x^{\prime}ABy^{\prime}BAz^{\prime})$
where $x^{\prime}$, $y^{\prime}$ and $z^{\prime}$ are words on $\mathcal{A}$
possibly including "$|$" character.
Moreover since $|A|_{ij} =\tau_{k}(|B|_{ij})= \tau(|B|)_{ij}$,
we obtain $|A|$ is equal to $\tau(|B|)$. Therefore we can apply the second
homotopy moves to $P_{1}$ : 
$$P_{1}=(x^{\prime}ABy^{\prime}BAz^{\prime}) \simeq_{S} 
(x^{\prime}y^{\prime}z^{\prime})=P_{2}.$$  
Consider the third homotopy move 
$$\varphi(P_{1})= (\mathcal{A}^{\prime},xAByACzBCt) \rightarrow 
\varphi(P_{2})=(\mathcal{A}^{\prime},xBAyCAzCBt)$$
where $(\|A\|,\|B\|,\|C\|) \in S_{\alpha,k}$, 
and $x$, $y$, $z$ and $t$ are words on 
$\mathcal{A}^{\prime}$.\par
By the definition of $S_{\alpha,k}$, we have
$\|A\|$ is equal to $|A|_{ij}$,
$\|B\|$ is equal to $|B|_{il}$ and
$\|C\|$ is equal to $|C|_{jl}$ for some $1 \le i \le \j \le l \le k$.
This implies 
$P_{1}$ form $(\mathcal{A},x^{\prime}ABy^{\prime}ACz^{\prime}BCt^{\prime})$
and
$P_{2}$ form $(\mathcal{A},x^{\prime}BAy^{\prime}CAz^{\prime}CBt^{\prime})$
where $x^{\prime}$, $y^{\prime}$, $z^{\prime}$ and $t^{\prime}$ 
are words on $\mathcal{A}$ possibly including "$|$" character.
Note that $(|A|,|B|,|C|) \in S$ by the definition of $S_{\alpha,k}$.
Thus we can apply the third homotopy move to $P_{1}$:
$$P_{1}= (\mathcal{A},x^{\prime}ABy^{\prime}ACz^{\prime}BCt^{\prime})
\simeq_{S} (\mathcal{A},x^{\prime}BAy^{\prime}CAz^{\prime}CBt^{\prime})
=P_{2}.$$
By the above $\varphi_{\bullet}$ is injection.\par
Now we completed the proof.
\end{proof}
We have Theorems \ref{thm1} and \ref{thm2} as corollaries of 
Theorems \ref{thmt1}, \ref{thmt2} and \ref{thm3}. \par
On the other hand, for a nanoword over $\alpha_{k}$,
the necessary and sufficient condition for existence of a corresponding
nanophrase over $\alpha$ is described as follows.
\begin{proposition}
Let $\alpha$ be an alphabet and $(\mathcal{A}^{\prime},w)$ 
be a nanoword over $\alpha_{k}$. 
Then there exists nanophrase of length $k$ over $\alpha$ such that
$(\mathcal{A},\varphi(P))$ is equal to $(\mathcal{A}^{\prime},w)$ 
if and only if $w$ satisfies 
conditions (1)-(4): Let $i_{A} := min\{w^{-1}(A)\}$ and 
$j_{A} := max\{w^{-1}(A)\}$ for each $A \in \mathcal{A}^{\prime}$, and 
$\|A\|=x_{m_{A}n_{A}}$ and $\|B\|=y_{m_{B}n_{B}}$ for some $x,y \in \alpha$.
Consider two letters $A, B \in \mathcal{A}^{\prime}$. Then \\
(1) If $i_{A} \le i_{B}$, then $m_{A} \le m_{B}$,\\  
(2) If $i_{A} \le j_{B}$, then $m_{A} \le n_{B}$,\\
(3) If $j_{A} \le i_{B}$, then $n_{A} \le m_{B}$,\\
(4) If $j_{A} \le j_{B}$, then $n_{A} \le n_{B}$.
\end{proposition}  
\begin{proof}
It is clear that $\varphi(P)$ satisfies the condition (1) - (4) 
for all $P \in \mathcal{P}_{k}(\alpha)$ by the definition of 
$\varphi$.\par
We show if $w \in \mathcal{N}(\alpha_{k})$ satisfies the condition 
(1) - (4), then there exist $P \in \mathcal{P}_{k}(\alpha)$ such that
$\varphi(P)$ is equal to $w$.\par
Suppose $w \in \mathcal{N}(\alpha_{k})$ satisfies the condition 
(1) - (4). Then we construct a nanophrase of length $k$ over 
$\alpha$ $P$ as follows.\par
(Step 1) We read $w$ from $w(1)$ and find a number $i$ 
which satisfies following conditions: Let
$\|w(i)\| = x_{m_{i}n_{i}}$ and $\|w(i+1)\|=y_{m_{i+1}n_{i+1}}$.
Then \\
$min\{w^{-1}(w(i))\} = i$, $min\{w^{-1}(w(i+1))\} = i+1$ 
and $m_{i} < m_{i+1}$ or \\ 
$min\{w^{-1}(w(i))\} = i$, $min\{w^{-1}(w(i+1))\} < i+1$ 
and $m_{i} < n_{i+1}$ or \\
$min\{w^{-1}(w(i))\} < i$, $min\{w^{-1}(w(i+1))\} = i+1$ 
and $n_{i} < m_{i+1}$ or \\
$min\{w^{-1}(w(i))\} < i$, $min\{w^{-1}(w(i+1))\} < i+1$ 
and $n_{i} < n_{i+1}$. \\
If there is not exists such $i$, then by the condition 
(1) - (4), for all $i,j,k,l \in \widehat{length(w)}$ 
$m_{i}$, $m_{j}$, $n_{i}$ and $n_{j}$ are equal.
We put this number $m$ and consider the nanophrase of length $k$ over $\alpha$,
$P=(\emptyset|\cdots|\emptyset|\stackrel{m}{\check{w}}|
\emptyset|\cdots|\emptyset)$. Then this is the required nanophrase.\\ 
(Step 2)
If there exists such $i$ (denote this $i$ by $i_{1}$), then 
we put the character $|$ into between $w(i)$ and $w(i+1)$.
Moreover if $(m_{i+1}-m_{i})$ ($(m_{i+1}-n_{i})$, $(n_{i+1}-m_{i})$ or  
$(n_{i+1}-n_{i})$) is grater than or equal to two, then 
we put $l-1$'s $\emptyset$ into between $w(i)$ and $w(i+1)$:
$\cdots w(i)|\emptyset|\cdots|\emptyset|w(i+1) \cdots$.\\
(Step 3)
We read $w$ from $w(i+1)$ and repeat (Step 1) and (Step 2) until the end.\\
(Step 4)
If $m_{1}$ is greater than or equal to two, 
then we put $m_{1}-1$'s $\emptyset$ in the left of $w(1)$:
$(\emptyset|\cdots|\emptyset|w(1)|\cdots$.
Moreover if $k-n_{length(w)}$ is greater than or equal to one,
then we put $k-n_{length(w)}$'s $\emptyset$ in the right of $w(length(w))$:
$\cdots w(length(w))|\emptyset|\cdots|\emptyset)$.\par
Then the obtained nanophrase is the required nanophrase over $\alpha$.
\end{proof}
Now we obtain following corollaries.
\begin{corollary}
The set $\{[w] \in \mathcal{N}((\alpha_{0})_{k},Q_{\alpha_{0},k}
R_{\alpha_{0},k},(S_{0})_{\alpha_{0},k})|w \ \rm{satisfies (1)-(4)}\}$ 
where $[w]$ is a homotopy class of $w$ is one to one corresponds to
the set $\mathcal{C}_{k}$.
\end{corollary}
\begin{corollary}
The set $\{[w] \in \mathcal{N}((\alpha_{\ast})_{k},Q_{\alpha_{\ast},k}
R_{\alpha_{\ast},k},(S_{\ast})_{\alpha_{\ast},k})|w \ \rm{satisfies (1)-(4)}\}$ 
where $[w]$ is a homotopy class of $w$ 
is one to one corresponds to
the set $\mathcal{L}_{k}$.
\end{corollary}
\begin{remark}\label{rem1}
In the case $\alpha$ is equal to $\alpha_{0}$ and $S$ is equal to
$S_{o}$, $\varphi_{\bullet}$ is the map
from $\mathcal{C}_{k}$ to $\mathcal{N}((\alpha_{0})_{k},$$ 
Q_{\alpha_{0},k},$$R_{\alpha_{0},k},$ $(S_{0})_{\alpha_{0},k})$ which maps 
$C \in  \mathcal{C}_{k}$ to $\varphi_{\bullet}(C)$ as follows.
Let us label the double points of $C$ by distinct letters
$A_1,\cdots,A_n$. Starting at the base point of the first component 
of $C$ and
following along $C$ in the direction of $C$, we write down the labels 
of double points which we passes until the return to the base point.
Then we obtain a word $w_{1}$ on the alphabet 
$\mathcal{A}=\{A_{1},\cdots,A_{n}\}$.
Similarly we obtain words $w_{j}$ on $\mathcal{A}$ from the
the $j$-th component for each $j \in \hat{k}$.
Let $t_{i}^{1}$ (respectively, $t_{i}^{2}$) be the tangent vector 
to $C$ at the double point
labeled $A_{i}$ appearing at the first (respectively, the second) passage
through this point. Set $|A_{i}|=a_{ij}$, if the pair $(t_i^1,t_i^2)$ is
positively oriented and $A_{i}$ is an intersection of the $i$-th component
and the $j$-th component of $C$, 
and  $|A_i|=b_{ij}$ if the pair $(t_i^1,t_i^2)$ is
negatively oriented and $A_{i}$ is an intersection of the $i$-th component
and the $j$-th component of $C$.
Then we obtain an $(\alpha_{0})_{k}$-alphabet $\mathcal{A}$. 
Finally we obtain the required
nanoword $\varphi(C):=(\mathcal{A},w_{1} w_{2} \cdots w_{k})$.
\end{remark}
Similarly as the Remark \ref{rem1}, we can regard an \emph{ornament} 
as an element of  
$\mathcal{N}$ $((\alpha_{0})_{k}$, $Q_{\alpha_{0},k}$, 
$R_{\alpha_{0},k}$, $S_{orn})$ where $S_{orn}$ is equal to 
$(S_{0})_{\alpha_{0},k}$ $\setminus$ $\{(c_{ij}, c_{il}, c_{jl})$ 
$|c \in \alpha_{0},$ $i \neq j \neq l \neq i \}$
(definition and explanations on ornaments is given in \cite{Va} for example).
In this point of view, we can say the $(Q,R,S)$-homotopy 
theory of nanowords is the common generalization of the 
stable equivalence theory of curves on surfaces and the theory of stable 
equivalence of 
ornaments. Moreover if we consider the case $Q$, $R$ and $S$ are  
equal to the empty set, then $(Q,R,S)$-homotopy classes of nanowords with 
satisfying the conditions (1) - (4) are one-to-one correspond to
the stably homeomorphic classes of ordered, pointed multi-component curves
on surfaces.\par
We can also say that to studying 
$(Q_{\alpha,k},R_{\alpha,k},S_{\alpha,k})$-homotopy of 
nanowords over $\alpha_{k}$ is natural since 
this is a natural generalization of the theory of curves on surfaces.  
\section{$(Q,R,S)$-homotopy Invariants Derived from $S$-homotopy Invariants.}
\label{ext}
In this section we construct $(Q_{\alpha,k},R_{\alpha,k},S_{k})$-homotopy
invariants for nanowords over $\alpha_{k}$ 
from some $S$-homotopy invariants for 
nanophrases of length $k$ over $\alpha$ via correspondence which described 
in the proof of Theorem \ref{thm3}. Moreover we show some of 
them are $S_{k}$-homotopy invariants for nanowords over $\alpha_{k}$.\par
In \cite{tu1} and \cite{tu2} V.Turaev constructed $S$-homotopy
invariants for a set $S \subset \alpha \times \alpha \times \alpha$ which
is strictly larger than the diagonal set (for example minimum length norm,
 $\alpha$-keis and $\alpha$-quandles for nanophrases, etc).
However those invariants are difficult to compute.
On the other hand, invariants which 
defined in this section are calculated easily.
Thus this is an available application of generalized homotopy theory
of words and phrases.
\subsection{Some Simple Invariants.}
First we extend the \emph{linking vector} of nanophrases 
(cf.\cite{fu2} and \cite{gi1}).\par
Let $(\mathcal{A},A_{1}A_{2}\cdots A_{2n})$ be a nanoword 
of length $2n$ over $\alpha_{k}$ where $\mathcal{A}$ is an 
$\alpha_{k}$-alphabet $\{A_{1}, A_{2} \cdots A_{2n}\}$ with  
a projection $\|A_{i}\|$ is equal to $|A_{i}|_{m_{i}n_{i}}$
where $|A_{i}| \in \alpha$. 
Let $\pi(\alpha,\tau)$ be a group which is generated by 
the all elements of $\alpha$ with relations $ab = ba$ and 
$a\tau(a)=1$ where $1$ is the unit element of $\pi(\alpha,\tau)$.
We denote a set 
$\{ A \in \mathcal{A} | \|A\| = x_{ij} \  \rm{for \ some \ x \in \alpha }\}$
by $\mathcal{A}_{ij}$.
Then we define a map 
$\widetilde{lk}:\mathcal{N}(\alpha_{k})$ 
$\longrightarrow \pi(\alpha,\tau)^{\frac{1}{2}n(n-1)}$ as follows:
$$\widetilde{lk}(w)=(\widetilde{l}_{w}(1,2),\widetilde{l}_{w}(1,3),
\cdots , \widetilde{l}_{w}(k-1,k)).$$
where 
$$\widetilde{l}_{w}(i,j) = \prod_{A \in \alpha, \|A\| \in \mathcal{A}_{ij}}|A|.$$
Note that for a nanoword $w$ over $\alpha_{k}$ which is the image of  
the map $\varphi$ (in other words, there exist a nanophrase $P$ of 
length $k$ over $\alpha$ such that $w$ is 
equal to $\varphi(P)$), $\widetilde{lk}(w)$ is coincides to $lk(P)$.
Therefore $\widetilde{lk}$ is an extension of $lk$.   
Then we obtain a following proposition.
\begin{proposition}
The $\widetilde{lk}$ is a $(Q_{\alpha,k},R_{\alpha,k},S_{\alpha,k})$-homotopy 
invariant for nanowords over $\alpha_{k}$.
\end{proposition}  
\begin{proof}
It is clear that an isomorphism does not change $\widetilde{lk}$.\par
Consider the first homotopy move 
$$w_{1}:=(\mathcal{A}, xAAy) \longrightarrow w_{2}:=(\mathcal{A} \setminus
\{A\}, xy)$$
where $\|A\| \in Q_{\alpha,k}$, 
$x$ and $y$ are words on $\mathcal{A}$.\par
In this case $A \in \mathcal{A}_{ii}$ does not contribute to
$\widetilde{lk}(w_{1})$ by the definition.\par
Consider the second homotopy move
$$w_{1}:=(\mathcal{A}, xAByBAz) \longrightarrow 
w_{2} := (\mathcal{A} \setminus \{A,B\}, xyz)$$
where $\|A\|$ is equal to $\tau_{k}(\|B\|)$, 
and $x$, $y$ and $z$ are words on $\mathcal{A}$.\par
Assume $\|A\|$ is equal to $|A|_{ij}$, then $\|B\|$ is equal to 
$\tau(|B|)_{ij}$.
It is sufficient to show that $\widetilde{l}_{w_{1}}(i,j)$ is 
equal to $\widetilde{l}_{w_{2}}(i,j)$. 
Note that $|A|$ is equal to $\tau(|B|)$ by the definition of $\| \cdot \|$.
Therefore,
\begin{eqnarray*}
\widetilde{l}_{w_{1}}(i,j) &=& \prod_{X \in \mathcal{A}_{ij}}|X| \\
                           &=& \prod_{X \in \mathcal{A}_{ij} \setminus \{A,B\}}
                                |X| \cdot |A| \cdot |B|\\
                           &=& \prod_{X \in \mathcal{A}_{ij} \setminus \{A,B\}}
                                |X|\cdot|A|\cdot \tau(|A|)\\
                           &=& \prod_{X \in \mathcal{A}_{ij} \setminus \{A,B\}}
                                |X|\\
                           &=& \widetilde{l}_{w_{2}}(i,j).
\end{eqnarray*} 
Thus $\widetilde{lk}(w_{1})$ is equal to $\widetilde{lk}(w_{2})$.\par
Consider the third homotopy move 
$$w_{1}:= (\mathcal{A},xAByACzBCt) \rightarrow 
w_{2}:=(\mathcal{A},xBAyCAzCBt)$$
where $(\|A\|,\|B\|,\|C\|) \in S_{\alpha,k}$, 
and $x$, $y$, $z$ and $t$ are words on $\mathcal{A}$.\par
Note that the third homotopy move does not change the projections of 
letters. Thus the third homotopy move does not change the value of 
$\widetilde{lk}$.\par
Now we completed the proof.   
\end{proof}   
Next we extend the \emph{component length vector} 
(see \cite{fu2} and \cite{gi1}). 
Let $v$ be a nanoword over $\alpha_{k}$.  
Then we define a vector $\widetilde{w}(v) \in (\mathbb{Z}/2\mathbb{Z})^{k}$
as follows.
$$\widetilde{w}(v)=
(\widetilde{w}(1),\widetilde{w}(2),\cdots ,\widetilde{w}(k)),$$
where $\widetilde{w}(l)$ $=$ $\sum_{i < l}$ $Card(\mathcal{A}_{il})$ $+$ 
 $\sum_{l < i}$ $Card(\mathcal{A}_{li})$ modulo $2$ for each 
$l \in \hat{k}$.       
Note that for a nanoword $v$ over $\alpha_{k}$ which is the image of  
the map $\varphi$ (in other words, there exist a nanophrase $P$ of 
length $k$ over $\alpha$ such that $w$ is 
equal to $\varphi(P)$), $\widetilde{w}(v)$ is coincides to $w(P)$.
Thus $\widetilde{w}$ is an extension of $w$.   
Then we obtain a following proposition.  
\begin{proposition}
The $\widetilde{w}$ is a $(Q_{\alpha,k},R_{\alpha,k},S_{\alpha,k})$-homotopy 
invariant for nanowords over $\alpha_{k}$.
\end{proposition}  
\begin{proof}
It is clear that isomorphisms does not change the
$\widetilde{w}$.\par
Consider the first homotopy move 
$$w_{1}:=(\mathcal{A}, xAAy) \longrightarrow 
w_{2}:=(\mathcal{A} \setminus \{A\}, xy)$$
where $\|A\| \in Q_{\alpha,k}$, 
$x$ and $y$ are words on $\mathcal{A}$.\par
In this case $A \in \mathcal{A}_{ii}$ does not contribute to
$\widetilde{w}(w_{1})$ by the definition.
Thus we obtain $\widetilde{w}(w_{1})=\widetilde{w}(w_{2})$.\par
Consider the second homotopy move
$$w_{1}:=(\mathcal{A}, xAByBAz) \longrightarrow 
w_{2} := (\mathcal{A} \setminus \{A,B\}, xyz)$$
where $\|A\|$ is equal to $\tau_{k}(\|B\|)$, 
and $x$, $y$ and $z$ are words on $\mathcal{A}$.\par
Let $\|A\| = |A|_{ij}$ and $\|B\| = |B|_{ij}$.
We need to consider the case $i$ is not equal to $j$.
In this case, both $A$ and $B$ are elements of $\mathcal{A}_{ij}$.
Thus $\widetilde{w}(l)$ changes $0$ or $2$ for each $l \in \hat{k}$.
Therefore $\widetilde{w}$ does not change by the second homotopy move.\par
Consider the third homotopy move 
$$w_{1}:= (\mathcal{A},xAByACzBCt) \rightarrow 
w_{2}:=(\mathcal{A},xBAyCAzCBt)$$
where $(\|A\|,\|B\|,\|C\|) \in S_{\alpha,k}$, 
and $x$, $y$, $z$ and $t$ are words on $\mathcal{A}$.\par
Note that the third homotopy move does not change the projections of 
letters. Thus the third homotopy move does not change the value of 
$\widetilde{w}$.\par
Now we completed the proof.   
\end{proof}
\begin{remark}
Note that $(Q_{\alpha,k},R_{\alpha,k},S_{\alpha,k})$-homotopy 
invariants $\widetilde{lk}$ and $\widetilde{w}$ are 
not $S_{\alpha,k}$-homotopy invariants. For example
consider nanowords over $\{a\}_{2}$, 
$w_{1}=AA$ with $\|A\|=a_{12}$ and $w_{2}= \emptyset$.
Then $\widetilde{lk}(w_{1})=a$. On the other hand 
$\widetilde{lk}(w_{2})=1$.
Similarly $\widetilde{w}(w_{1})=(1,1)$ and $\widetilde{w}(w_{2})=(0,0)$.
Thus $\widetilde{lk}$ and $\widetilde{w}$ are not 
$S_{\alpha,k}$-homotopy invariants.  
\end{remark}
\subsection{The invariant $\widetilde{S_{o}}$.}
In the paper \cite{gi1}, A.Gibson defined a homotopy invariant
for Gauss phrases (in other words, a homotopy invariant for nanophrases
over the one-element set) which is called $S_{o}$.
Moreover Gibson showed that Gibson's $S_{o}$ invariant is 
strictly stronger than the invariant $T$ for Gauss phrases
which is defined in \cite{fu1} (see also \cite{fu2}).
In this section, we extend the $S_{o}$ invariant to the 
$(S_{\alpha,k},R_{\alpha,k},S_{\alpha,k})$-homotopy invariant for 
nanowords over $\alpha_{k}$ via correspondence in the proof of 
Theorem \ref{thm3}. Moreover we show the obtained invariant 
is also an $S_{\alpha,k}$-homotopy invariant. To do this, first
we extend Gibson's $S_{o}$ invariant for Gauss phrases to
a homotopy invariant for nanophrases over \emph{any} $\alpha$.  
\subsubsection{An extension of Gibson's $S_{o}$ invariant.}
In this sub-subsection we extend  
Gibson's $S_{o}$ invariant for Gauss phrases to
a homotopy invariant for nanophrases over any alphabet.\par
Let $\alpha$ be an alphabet with an involution 
$\tau:\alpha \rightarrow \alpha$.
Since the set $\alpha$ is a finite set,
we obtain following orbit decomposition of the $\tau$ : 
$\alpha/\tau$ $=$ $\{ \widehat{a_{i_1}}$, $\widehat{a_{i_2}}$, $\cdots 
,\widehat{a_{i_l}},$ $\widehat{a_{i_{l+1}}},$ $\cdots,$ 
$\widehat{a_{i_{l+m}}}\}$, where 
$\widehat{a_{i_j}}:= \{ a_{i_j}, \tau(a_{i_j}) \}$ such that 
$Card(\widehat{a_{i_j}})=2$ for all $j \in \{1,\cdots,l \}$ and 
$Card(\widehat{a_{i_j}})=1$ for all $j \in \{l+1,\cdots,l+m \}$ (we fix
a complete representative system $\{ a_{i_1}, a_{i_2}, \cdots 
,a_{i_l}, a_{i_{l+1}},\cdots, a_{i_{l+m}} \}$ which satisfy the above
condition). We denote a complete representative system which satisfies
above condition $crs(\alpha/\tau)$.
Let $\mathcal{A}$ be a $\alpha$-alphabet.  
For $A \in \mathcal{A}$ we define $\varepsilon(A) \in \{ \pm 1 \} $ by
$$\varepsilon(A):=\begin{cases} 1 \ ( \ if \ |A|=a_{i_j} \ for \ some \
		  j \in \{1,\cdots l+m \} \  ),
 \\ -1 \ ( \ if \ |A|=\tau({a_{i_j}}) \ for \ some \ j \in \{1,\cdots
		   l\} \ ). 
\end{cases}$$
Let $P=(\mathcal{A},(w_1|\cdots|w_k))$ be a nanophrase
over $\alpha$ and $A$ and $B$ be letters in $\mathcal{A}$. 
Let $K_{(i,j)}$ be $\mathbb{Z}$ if $i \le l$ and $j \le l$, 
otherwise $\mathbb{Z} / 2\mathbb{Z}$. 
We denote $K_{(1,1)} \times K_{(1,2)} \times \cdots K_{(1,l+m)} \times K_{(2,1)}
\times \cdots \times K_{(l+m,l+m)}$ by $\prod K_{(i,j)}$. 
Then we define
$\sigma_{p}^{j}(A,B) \in \prod {K_{(i,j)}}$ 
as follows:  
If $A$ and $B$ form
$\cdots A \cdots B \cdots A \cdots B \cdots$ in $P$, 
the second $B$ appears in the $j$-th component of $P$,   
$|A| \in \widehat{a_{i_{p}}}$ and $|B|=a_{i_q}$
for some $p,q \in \{1, \cdots l+m \}$, or 
$\cdots B \cdots A \cdots B \cdots A \cdots$ in $P$, 
the first $B$ appears in the $j$-th component of $P$, 
$|A| \in \widehat{a_{i_{p}}}$ and
$|B|=\tau(a_{i_q})$ for some $p,q \in \{1, \cdots l+m\}$ , then 
$\sigma_{P}^{j}(A,B):= \mathbf{e}_{(p,q)}$
If $\cdots A \cdots B \cdots A \cdots B \cdots$ in $P$,
the second $B$ appears in the $j$-th component of $P$, 
$|A| \in \widehat{a_{i_{p}}}$ and  $|B|=\tau(a_{i_q})$, 
or $\cdots B \cdots A \cdots B \cdots A \cdots$ in $P$,
the first $B$ appears in the $j$-th component of $P$, 
$|A| \in \widehat{a_{i_{p}}}$ and $|B|=a_{i_q}$, then 
$\sigma_{P}^{j}(A,B):=-\mathbf{e}_{(p,q)}$. 
Otherwise $\sigma_{P}^{j}(A,B):=\mathbf{0}$,
where 
$\mathbf{e}_{(p,q)}$ $=$ $(0$ $,\cdots,0$, 
$\stackrel{(p,q)}{\check{1}},0,$ $\cdots,0).$ 
Moreover we define notations $l_{P,j}$, $l_{P}$ as follows:
$$l_{P,j}(A) = \sum_{X \in \mathcal{A}} \sigma_{P}^{j}(A,X) \in \prod K_{(p,q)},$$
$$l_{P}(A) = (l_{P,1}(A),l_{P,2}(A), \cdots ,l_{P,k}(A)) 
\in (\prod K_{(p,q)})^{k}.$$
Furthermore we define a map $(B_{P})_{j}$ from 
$(\prod K_{(p,q)})^{k}$ $\setminus$
$\{\mathbf{0}\}$ to $\mathbb{Z}$
or $\mathbb{Z}/2\mathbb{Z}$ as follows.
First we define \emph{type} of elements in $(\prod K_{(p,q)})^{k}$.
Let $\mathbf{v}$ be an element of $(\prod K_{(p,q)})^{k}$
and $(\mathbf{v})_{i,r,s}$ be an element in $\mathbf{v}$ corresponding
to i th. $K_{(r,s)}$ in  $(\prod K_{(p,q)})^{k}$. 
Then $\mathbf{v}$ is \emph{type (i)} if $(\mathbf{v})_{i,r,s}$ is 
not equal to $0$ only if $r$ is less than or equal to $l$ for 
all $i$, $r$, $s$. 
We say $\mathbf{v}$ is \emph{type (ii)} if $(\mathbf{v})_{i,r,s}$ is 
not equal to $0$ only if $r$ is grater than $l$ for 
all $i$, $r$, $s$. Otherwise we call  $\mathbf{v}$ \emph{type (iii)}.
Then 
$$(B_{P})_{j}(\mathbf{v}) = \begin{cases} \sum_{A \in \mathcal{A}_{j}, l_{P}(A) = \mathbf{v}}
\varepsilon(A) 
\ ( \ if \ \mathbf{v} \ is \ type(i) \ ),
 \\ \sum_{A \in \mathcal{A}_{j}, l_{P}(A)=\mathbf{v}}\varepsilon(A) \ (\ mod \ 2 \ ) 
\ ( \ if \ \mathbf{v} \ is \ type(ii) \ ), \\
0 \ ( \ otherwise \ ).
\end{cases}$$     
Then we define $S_{o}$ as follows: 
$$S_{o}(P) = ((B_{P})_{1}, (B_{P})_{2},\cdots,(B_{P})_{k}).$$ 
\begin{remark}
If we consider the case $\alpha$ is equal to the one-element set,
then $S_{o}$ is coincides to the Gibson's $S_{o}$ invariant in \cite{gi1}.
See \cite{gi1} for more details. 
\end{remark}
\begin{remark}
After submitting previous version of this paper to arXiv, 
Andrew Gibson gave me
many useful comments on the invariant $S_{o}$
which was defined in the previous version of this paper.
According to his advice, the author modified the definition 
of the extended $S_{o}$ invariant. 
Then the extended $S_{o}$ invariant in this version became stronger, 
and this invariant $S_{o}$ became equivalent to the invariant $U$ 
which is defined independently in \cite{gi2} by Gibson.    
\end{remark}
\begin{proposition}\label{propSo}
The $S_{o}$ is a homotopy invariant for nanophrases over $\alpha$.
\end{proposition}
\begin{proof}
It is sufficient to show homotopy invariance of $(B_{P})_{j}$.\par
It is clear that $S_{o}$ does not change under isomorphisms.\par
Consider the first homotopy move 
$$P_{1}:=(\mathcal{A}, (xAAy)) \longrightarrow P_{2}:=(\mathcal{A} \setminus
\{A\}, (xy))$$
where $x$ and $y$ are words on $\mathcal{A}$, 
possibly including "$|$" character.\par
For each $\mathcal{A}$ and $j \in \hat{k}$, 
$\sigma_{P_{1}}^{j}(A,X)$ is equal to $0$ and 
$\sigma_{P_{1}}^{j}(X,A)$ is equal to $0$ for all $X \in \mathcal{A}$. 
Thus $A$ does not contribute to $l_{P}(X)$ and $l_{P}(A)$ is equal to $0$.
By the definition of $(B_{P_{1}})_{j}$, 
$A$ does not contribute to $(B_{P_{1}})_{j}$.
Therefore $(B_{P})_{j}$ is an invariant under the first homotopy move.\par
Consider the second homotopy move
$$P_{1}:=(\mathcal{A}, (xAByBAz)) \longrightarrow 
P_{2} := (\mathcal{A} \setminus \{A,B\}, (xyz))$$
where $|A|$ is equal to $\tau(|B|)$, 
and $x$, $y$ and $z$ are words on $\mathcal{A}$
possibly including "$|$" character.\par
We show $l_{P_{1}}(D)$ is equal to $l_{P_{2}}(D)$ for all $D$ $\in$ 
$\mathcal{A}$ $\setminus$ $\{A,B\}$.
In fact, 
\begin{eqnarray*}
l_{P_{1}}(D) &=& \sum_{X \in \mathcal{A}}\sigma_{P_{1}}^{j}(D,X) \\
            &=& \sum_{X \in \mathcal{A} \setminus \{A,B\}}\sigma_{P_{1}}^{j}(D,X)
                + \sigma_{P_{1}}^{j}(D,A) + \sigma_{P_{1}}^{j}(D,B) 
                 \ \  \cdots\cdots (\ast)
\end{eqnarray*}
If $\widehat{|D|}$ is a fixed point of $\tau$, then all non-zero entries of 
$\sigma_{P_{1}}^{j}(D,A)$ and $\sigma_{P_{1}}^{j}(D,B)$ are elements
of $\mathbb{Z}/2\mathbb{Z}$ and 
$\sigma_{P_{1}}^{j}(D,A)$ is equal to $\sigma_{P_{1}}^{j}(D,B)$. 
Thus $\sigma_{P_{1}}^{j}(D,A)+\sigma_{P_{1}}^{j}(D,B)=0$.\par
If $\widehat{|D|}$ is a not fixed point of $\tau$ 
(in other words, $\widehat{|D|}$ is a  
\emph{free orbit} of $\tau$) and $\widehat{|A|}(=\widehat{|B|})$ is a fixed 
point of $\tau$, then we obtain 
$\sigma_{P_{1}}^{j}(D,A)+\sigma_{P_{1}}^{j}(D,B)=0$
similarly as the above case.\par
If $\widehat{|D|}$ and $\widehat{|A|}$ are a free orbits of $\tau$,
then $\sigma_{P_{1}}^{j}(D,A)$ is equal to $-\sigma_{P_{1}}^{j}(D,B)$.
Thus we obtain 
$\sigma_{P_{1}}^{j}(D,A)+\sigma_{P_{1}}^{j}(D,B)=0$.
Therefore  
\begin{eqnarray*}
(\ast) &=& \sum_{X \in \mathcal{A} \setminus \{A,B\}}\sigma_{P_{1}}^{j}(D,X)\\
       &=& \sum_{X \in \mathcal{A} \setminus \{A,B\}}\sigma_{P_{2}}^{j}(D,X)\\
       &=& l_{P_{2},j}(D).
\end{eqnarray*}
If $A$ and $B$ are not in $\mathcal{A}_{j}$ for all $j \in \hat{k}$, 
then we completed the proof of the proposition.\par
Now we assume $A \in \mathcal{A}_{j}$ for some $j \in \hat{k}$.
Then $B \in \mathcal{A}_{j}$ since $\widehat{|A|}$ is equal to
$\widehat{|B|}$. Moreover by arrangement of letters $A$ and $B$, 
we obtain $l_{P_{1}}(A)$ is equal to $l_{P_{2}}(A)$.
Thus 
$$(B_{P_{1}})_{j}(l_{P_{1}}(A)) = 
\sum_{X \in \mathcal{A}_{j} \setminus \{A,B\},l_{P_{1}}(X)=l_{P_{1}}(A)}
\varepsilon(X)+\varepsilon(A)+\varepsilon(B)$$
If $\widehat{|A|}$ is a free orbit of $\tau$, then $\varepsilon(A)$ is equal to
$-\varepsilon(B)$ since $|A|$ is equal to $\tau(|B|)$.
Thus $\varepsilon(A) + \varepsilon(B) = 0$.
If $\widehat{|A|}$ is a fixed point of $\tau$, then $l_{P_{1}}(A)$ is 
type (ii) and 
$\varepsilon(A) + \varepsilon(B) = 2 = 0$ (mod $2$).  
Thus contributions of $A$ and $B$ to $(B_{P_{1}})_{j}$ is vanish.
It is clear that $(B_{P_{1}})_{j}(\mathbf{v})$ is equal to 
$(B_{P_{2}})_{j}(\mathbf{v})$ if $\mathbf{v}$ is not equal to $l_{P_{1}}(A)$.
Therefore $(B_{P})_{j}$ is invariant under the second homotopy move.\par 
Consider the third homotopy move
$$P_{1}:= (\mathcal{A},(xAByACzBCt)) \rightarrow 
P_{2}:=(\mathcal{A},(xBAyCAzCBt))$$
where $|A|=|B|=|C|$, and $x$, $y$, $z$ and $t$ are words on 
$\mathcal{A}$ possibly including "$|$" character.\par
We call a letter $A$ \emph{single component letter} if 
$A$ appears twice in the same component of the phrase.
First we consider the case $A$, $B$ and $C$ are single component letters.
We show $l_{P_{1}}(X)$ is equal to $l_{P_{2}}(X)$ for all $X \in \mathcal{A}$.
For all $D \in \mathcal{A}$ $\setminus$ $\{A,B\}$,
it is clear that  $l_{P_{1}}(D)$ is equal to $l_{P_{2}}(D)$.\par
On $l_{P_{1}}(A)$ is equal to $l_{P_{2}}(A)$. Note that letters 
which contribute to $l_{P_{1}}(A)$ (respectively $l_{P_{2}}(A)$)
are $B$ and letters in $y$ (respectively  $C$ and letters in $y$).
It is easily checked that $\sigma_{P_{1}}^{j}(A,X)$ is equal to 
 $\sigma_{P_{2}}^{j}(A,X)$ for all letter $X$ in $y$.
Moreover since $\cdots$ $A$ $\cdots B$ $\cdots A$ $\cdots B$ 
$\cdots$ in $P_{1}$,
$\cdots$ $A$ $\cdots C$ $\cdots A$ $\cdots C$ 
$\cdots$ in $P_{2}$, 
$|B|$ is equal to $|C|$ 
and $B$ and $C$ that appear second times
are belong to the same component, we obtain
$l_{P_{1}}(A)$ is equal to $l_{P_{2}}(A)$.\par
On $l_{P_{1}}(B)$ is equal to $l_{P_{2}}(B)$.
Note that letters 
which contribute to $l_{P_{1}}(B)$ (respectively $l_{P_{2}}(B)$)
are letters in $y$, $A$, $C$ and letters in $z$ 
(respectively letters in $y$ and letters in $z$).
It is clear that  $\sigma_{P_{1}}^{j}(B,X)$ is equal to 
 $\sigma_{P_{2}}^{j}(B,X)$ for all letter $X$ in $y$ or $z$.
We show $\sigma_{P_{1}}^{j}(B,A) + \sigma_{P_{1}}^{j}(B,C)$ $=$ $0$.
Consider the case $\widehat{|A|}$ ($=\widehat{|B|}=\widehat{|C|}$) 
is a free orbit of $\tau$. Note 
that letters $A$, $B$ and $C$ are contained in the same component 
of the phrase (we suppose this component is the $j$-th component).
If we assume $\sigma_{P_{1}}^{j}(B,A)$ is equal to $\pm \mathbf{e}_{(j,j)}$,
then $\sigma_{P_{1}}^{j}(B,C)$ is equal to $\mp \mathbf{e}_{(j,j)}$ 
where the double signs are taken in the same order. Thus we obtain
$\sigma_{P_{1}}^{j}(B,A) + \sigma_{P_{1}}^{j}(B,C)$ $=$ $0$. 
Next consider the case $\widehat{|A|}$ ($=\widehat{|B|}=\widehat{|C|}$) 
is a fixed point of $\tau$.
Then $\sigma_{P_{1}}^{j}(B,A)$ is equal to $\mathbf{e}_{(j,j)}$ and
$\sigma_{P_{1}}^{j}(B,C)$ is equal to $\mathbf{e}_{(j,j)}$.
Since all the non-zero entries of $\sigma_{P_{1}}^{j}(B,A)$ and 
$\sigma_{P_{1}}^{j}(B,C)$ are the elements of $\mathbb{Z} / 2\mathbb{Z}$,
we obtain $\sigma_{P_{1}}^{j}(B,A) + \sigma_{P_{1}}^{j}(B,C)$ $=$ $0$.
Therefore 
\begin{eqnarray*}
l_{P_{1},j}(B) &=& \sum_{X \in \mathcal{A}}\sigma_{P_{1}}^{j}(B,X)\\
              &=& \sum_{X \in \mathcal{A} \setminus \{A,C\}}\sigma_{P_{1}}^{j}(B,X)+
                   \sigma_{P_{1}}^{j}(B,A) + \sigma_{P_{1}}^{j}(B,C)\\
              &=& \sum_{X \in \mathcal{A} \setminus \{A,C\}}
                                           \sigma_{P_{1}}^{j}(B,X)\\
              &=& \sum_{X \in \mathcal{A} \setminus \{A,C\}}
                                            \sigma_{P_{2}}^{j}(B,X)\\
              &=& l_{P_{2},j}(B)
\end{eqnarray*}
for each $j \in \hat{k}$.
Thus we obtain $l_{P_{1}}(B)$ is equal to $l_{P_{2}}(B)$.\par
On $l_{P_{1}}(C)$ is equal to $l_{P_{2}}(C)$.   
Note that letters 
which contribute to $l_{P_{1}}(C)$ (respectively $l_{P_{2}}(C)$)
are letters in $z$ and $B$ 
(respectively letters in $z$ and $A$). 
It is clear that $\sigma_{P_{1}}^{j}(C,X)$ is equal to 
$\sigma_{P_{2}}^{j}(C,X)$ for all letter $X$ in $z$.
Moreover since $\cdots$ $B$ $\cdots C$ $\cdots B$ $\cdots C$ 
$\cdots$ in $P_{1}$,
$\cdots$ $A$ $\cdots C$ $\cdots A$ $\cdots C$ 
$\cdots$ in $P_{2}$, 
$|A|$ is equal to $|C|$ 
and $A$ and $C$ that appear first times
are belong to the same component, we obtain
$l_{P_{1}}(C)$ is equal to $l_{P_{2}}(C)$.\par
Therefore if $A$, $B$ and $C$ are single component letters,
$(B_{P})_{j}$ is the invariant of the third homotopy move.\par
By using a part of above discussion,  
we can prove the case $A$, $B$ and $C$ are not single component letters.
Now we have completed the proof of the proposition. 
\end{proof}
In \cite{gi1}, Gibson showed that Gibson's $S_{o}$ invariant is 
strictly stronger than the invariant $T$ for Gauss phrases.
We have a similar statement for extended $S_{o}$ invariant for nanophrases
over any alphabet.
\begin{proposition}
The invariant $S_{o}$ for nanophrases over $\alpha$ is strictly
stronger than the invariant $T$ for nanophrases over $\alpha$.
\end{proposition}
\begin{proof}
It is sufficient to show that we can recover $T_{P}(w_{j})$ from
$(B_{P})_{j}$ (we use the same symbols as \cite{fu2}). 
First we define a set $B_{j}(P)$ as follows.  
\begin{eqnarray*}
B_{j}(P) &:=& \{ (B_{P})_{j}(\mathbf{v})\mathbf{v} | 
\exists A  \in \mathcal{A}_{j}   
                                   \ such \ that \ l_{P}(A) = \mathbf{v} \} \\
      &=& \{  (B_{P})_{j}(l_{P}(A))l_{P}(A) | A \in \mathcal{A}_{j} \}. \end{eqnarray*}
We split $\mathcal{A}_{j}$ into a disjoint sum of sets as follows.
$$\{A \in \mathcal{A}_{j} | l_{P}(A) = \mathbf{v}_{1} \} 
(=: \{A_{1}^{1},\cdots,A_{i_{1}}^{1}\}),$$ 
$$\{A \in \mathcal{A}_{j}| l_{P}(A) = \mathbf{v}_{2} \} 
(=: \{A_{1}^{2},\cdots,A_{i_{2}}^{2}\}),$$
$$\vdots$$
$$\{A \in \mathcal{A}_{j}| l_{P}(A) = \mathbf{v}_{m} \} 
(=: \{A_{1}^{m},\cdots,A_{i_{m}}^{m}\}).$$
Then
\begin{eqnarray*}
B_{j}(P) &=& \{ \sum_{j=1}^{i_{1}}\varepsilon(A_{j})\mathbf{v_{1}},\cdots,  
                        \sum_{j=1}^{i_{m}}\varepsilon(A_{j})\mathbf{v_{m}} 
                                                                     \} \\
      &=& \{ \sum_{j=1}^{i_{1}}\varepsilon(A_{j})l_{P}(A_{j}^{1}),\cdots,  
                       \sum_{j=1}^{i_{m}}\varepsilon(A_{j})l_{P}(A_{j}^{m})
                                                                    \}.   
\end{eqnarray*}
We sum up all the elements of $B_{j}(P)$, we obtain 
\begin{eqnarray*}
& & \sum_{A \in \mathcal{A}_{j}}\varepsilon(A)l_{P}(A)\\
&=& \left( \sum_{A \in \mathcal{A}_{j}}\varepsilon(A)
                      \sum_{X \in \mathcal{A}}\sigma_{P}^{1}(A,X),\cdots,
        \sum_{A \in \mathcal{A}_{j}}\varepsilon(A)
                      \sum_{X \in \mathcal{A}}\sigma_{P}^{k}(A,X) \right).
\end{eqnarray*}
Moreover we sum up all the entries of the above vector, we obtain 
\begin{eqnarray*}
\sum_{j=1}^{k} \sum_{A \in \mathcal{A}_{j}} \varepsilon(A) \sum_{X \in \mathcal{A}}
\sigma_{P}^{j}(A,X)
&=& 
\sum_{A \in \mathcal{A}_{j}} \varepsilon(A) \sum_{j=1}^{k} 
\sum_{X \in \mathcal{A}} \sigma_{P}^{j}(A,X) \\
&=& 
\sum_{A \in \mathcal{A}_{j}} \varepsilon(A) \sum_{X \in \mathcal{A}}
\sigma_{P}(A,X) \\
&=& 
T_{P}(w_{j}).
\end{eqnarray*}
Thus we can recover $T_{P}(w_{j})$ from $B_{j}(P)$. 
Therefore $S_{o}$ is strictly stronger than $T$
(the proof of "strictly" is shown in \cite{gi1} Example 6.4). 
\end{proof} 
In the next sub-subsection we extends the invariant $S_{o}$ for nanophrases 
over $\alpha$ to a $(Q_{\alpha,k},R_{\alpha,k},S_{\alpha,k})$-homotopy 
invariant  $\widetilde{S_{o}}$ for nanowords over $\alpha_{k}$.
Moreover we show $\widetilde{S_{o}}$ is an $S_{\alpha,k}$-homotopy invariant
for nanowords over $\alpha_{k}$. 
\subsubsection{The invariant $\widetilde{S_{o}}$.}
In this sub-subsection we extend the $S_{o}$ invariant which was defined in 
the previous sub-subsection 
to the invariant $\widetilde{S_{o}}$
for nanophrases over $\alpha_{k}$.
To define $\widetilde{S_{o}}$, we prepare some notations. 
For a nanoword $w$ over $\alpha_{k}$ and letters $A, B \in \mathcal{A}$,
we define $\widetilde{\sigma}_{w}^{j}(A,B)$ as follows:
If $A$ and $B$ form
$\cdots A \cdots B \cdots A \cdots B \cdots$ in $w$,   
$|A| \in \widehat{a_{p}}$ and $\|B\|=(a_{q})_{ij}$ 
for some $p, q \in \{1, \cdots l+m \}$ and $i \le j$, or 
$\cdots B \cdots A \cdots B \cdots A \cdots$ in $w$ , 
$|A| \in \widehat{a_{p}}$ and
$\|B\|=\tau(a_{q})_{jl}$ for some $p,q \in \{1, \cdots l+m\}$ and $l \ge j$, 
then 
$\widetilde{\sigma}_{w}^{j}(A,B):= \mathbf{e}_{(p,q)}$.
If $\cdots A \cdots B \cdots A \cdots B \cdots$ in $w$, 
$|A| \in \widehat{a_{p}}$ and  $\|B\|=\tau(a_{q})_{ij}$, 
for some $p,q \in \{1,\cdots,l+m\}$ and $i \le j$, 
or $\cdots B \cdots A \cdots B \cdots A \cdots$ in $w$, 
$|A| \in \widehat{a_{p}}$ and $\|B\|=(a_{i_q})_{jl}$
for some $p,q \in \{1,\cdots,l+m\}$ and $j \le l$, 
then 
$\widetilde{\sigma}_{P}^{j}(A,B):=-\mathbf{e}_{(p,q)}$. 
Otherwise $\widetilde{\sigma}_{P}^{j}(A,B):=\mathbf{0}$. 
Moreover we define notations $\widetilde{l}_{w,j}$, $\widetilde{l}_{w}$ 
and $(\widetilde{B}_{w})_{j}$ as follows:
$$\widetilde{l}_{w,j}(A) = \sum_{X \in \mathcal{A}} 
\widetilde{\sigma}_{w}^{j}(A,X) \in \prod K_{(p,q)},$$
$$\widetilde{l}_{w}(A) = (\widetilde{l}_{w,1}(A),\widetilde{l}_{w,2}(A), 
\cdots, \widetilde{l}_{w,k}(A)) 
\in (\prod K_{(p,q)})^{k}.$$
Then $(\widetilde{B}_{w})_{j}$ is a map from $ (\prod K_{(p,q)})^{k}$ 
$\setminus$ $\{\mathbf{0}\}$ to $\mathbb{Z}$ or $\mathbb{Z}/2\mathbb{Z}$ which is 
defined by 
$$(\widetilde{B}_{w})_{j}(\mathbf{v}) = \begin{cases} \sum_{A \in \mathcal{A}_{jj}, \widetilde{l}_{w}(A) = \mathbf{v}}
\varepsilon(A) 
\ ( \ if \ \mathbf{v} \ is \ type(i) \ ),
 \\ \sum_{A \in \mathcal{A}_{jj}, \widetilde{l}_{w}(A)=\mathbf{v}}\varepsilon(A) \ (\ mod \ 2 \ ) 
\ ( \ if \ \mathbf{v} \ is \ type(ii) \ ), \\
0 \ ( \ otherwise \ ).
\end{cases}$$  
Then we define $\widetilde{S}_{o}$ as follows: 
$$\widetilde{S}_{o}(w) = ((\widetilde{B}_{w})_{1}, (\widetilde{B}_{w})_{2},
\cdots,(\widetilde{B}_{w})_{k}).$$ 
Then we obtain the following proposition.
\begin{proposition}
The $\widetilde{S}_{o}$ is a $(Q_{\alpha,k},R_{\alpha,k},
(\Delta_{\alpha^{3}})_{\alpha,k})$-homotopy invariant of nanowords over 
$\alpha_{k}$.
\end{proposition}
\begin{proof}
It is sufficient to show homotopy invariance of $(\widetilde{B}_{w})_{j}$.\par
It is clear that $\widetilde{S}_{o}$ does not change under isomorphisms.\par
Consider the first homotopy move 
$$w_{1}:=(\mathcal{A}, xAAy) \longrightarrow w_{2}:=(\mathcal{A} \setminus
\{A\}, xy)$$
where $\|A\| \in Q_{\alpha,k}$, $x$ and $y$ are words on $\mathcal{A}$.\par
For each $\mathcal{A}$ and $j \in \hat{k}$, 
$\widetilde{\sigma}_{w_{1}}^{j}(A,X)$ is equal to $0$ and 
$\widetilde{\sigma}_{w_{1}}^{j}(X,A)$ is equal to $0$ for all
$X \in \mathcal{A}$. 
Thus $A$ does not contribute to $\widetilde{l}_{w}(X)$ and 
$\widetilde{l}_{w}(A)$ is equal to $0$.
By the definition of $(\widetilde{B}_{w_{1}})_{j}$, $A$ does not contribute to 
$(\widetilde{B}_{w_{1}})_{j}$.
Therefore $(\widetilde{B}_{w_{1}})_{j}$ is an invariant 
under the first homotopy move.\par
Consider the second homotopy move
$$w_{1}:=(\mathcal{A}, xAByBAz) \longrightarrow 
w_{2} := (\mathcal{A} \setminus \{A,B\}, xyz)$$
where $\|A\|$ is equal to $\tau_{k}(\|B\|)$ 
(put $\|A\|=|A|_{ij}$ and $\|B\|=|B|_{ij}$), 
and $x$, $y$ and $z$ are words on $\mathcal{A}$ \par
We show $\widetilde{l}_{w_{1}}(D)$ is equal to 
$\widetilde{l}_{w_{2}}(D)$ for all $D$ $\in$ 
$\mathcal{A}$ $\setminus$ $\{A,B\}$.
In fact, 
\begin{eqnarray*}
\widetilde{l}_{w_{1}}(D) &=& \sum_{X \in \mathcal{A}}
                               \widetilde{\sigma}_{w_{1}}^{j}(D,X) \\
                        &=& \sum_{X \in \mathcal{A} \setminus \{A,B\}}
                               \widetilde{\sigma}_{w_{1}}^{j}(D,X)
                            + \widetilde{\sigma}_{w_{1}}^{j}(D,A) + 
                              \widetilde{\sigma}_{w_{1}}^{j}(D,B) 
                            \ \  \cdots\cdots (\ast)
\end{eqnarray*}
If $\widehat{|D|}$ is a fixed point of $\tau$, then all non-zero entry of 
$\widetilde{\sigma}_{w_{1}}^{j}(D,A)$ and 
$\widetilde{\sigma}_{w_{1}}^{j}(D,B)$ are elements
of $\mathbb{Z}/2\mathbb{Z}$ and 
$\widetilde{\sigma}_{w_{1}}^{j}(D,A)$ is equal to 
$\widetilde{\sigma}_{w_{1}}^{j}(D,B)$. 
Thus 
$\widetilde{\sigma}_{w_{1}}^{j}(D,A)+\widetilde{\sigma}_{w_{1}}^{j}(D,B)=0$.\par
If $\widehat{|D|}$ is a free orbit of $\tau$ 
and $\widehat{|A|}(=\widehat{|B|})$ is a fixed 
point of $\tau$, then we obtain 
$\widetilde{\sigma}_{w_{1}}^{j}(D,A)+\widetilde{\sigma}_{w_{1}}^{j}(D,B)=0$
similarly as the above case.\par
If $\widehat{|D|}$ and $\widehat{|A|}$ are a free orbits of $\tau$,
then $\widetilde{\sigma}_{w_{1}}^{j}(D,A)$ is equal to 
$-\widetilde{\sigma}_{w_{1}}^{j}(D,B)$.
Thus we obtain 
$\widetilde{\sigma}_{w_{1}}^{j}(D,A)+\widetilde{\sigma}_{w_{1}}^{j}(D,B)=0$.
Therefore
\begin{eqnarray*}
(\ast) &=& \sum_{X \in \mathcal{A} \setminus \{A,B\}}
                           \widetilde{\sigma}_{w_{1}}^{j}(D,X)\\
       &=& \sum_{X \in \mathcal{A} \setminus \{A,B\}}
                           \widetilde{\sigma}_{w_{2}}^{j}(D,X)\\
       &=& \widetilde{l}_{w_{2},j}(D).
\end{eqnarray*}
If $A$ and $B$ are not in $\mathcal{A}_{jj}$ for all $j \in \hat{k}$, 
then we completed the proof of the proposition.\par
Now we assume $A \in \mathcal{A}_{jj}$ for some $j \in \hat{k}$.
Then $B \in \mathcal{A}_{jj}$ since $\widehat{\|A\|}$ is equal to
$\widehat{\|B\|}$ (note that $\tau_{k}$ does not change the subscript
of the projection). 
Moreover by arrangement of letters $A$ and $B$, 
we obtain $\widetilde{l}_{w_{1}}(A)$ 
is equal to $\widetilde{l}_{w_{2}}(A)$.
Thus 
$$(\widetilde{B}_{w_{1}})_{j}(\widetilde{l}(A)) = 
\sum_{X \in \mathcal{A}_{jj} \setminus \{A,B\},\widetilde{l}_{w_{1}}(X)
                                                 =\widetilde{l}_{w_{1}}(A)}
\varepsilon(X)+\varepsilon(A)+\varepsilon(B).$$
If $\widehat{|A|}$ is a free orbit of $\tau$, then $\varepsilon(A)$ is equal to
$-\varepsilon(B)$ since $|A|$ is equal to $\tau(|B|)$.
Thus $\varepsilon(A) + \varepsilon(B) = 0$.\par
If $\widehat{|A|}$ is a fixed point of $\tau$, then
$\varepsilon(A) + \varepsilon(B) = 2$ $= 0$ (mod $2$).
Thus contributions of $A$ and $B$ to 
$(\widetilde{B}_{w_{1}})_{j}(\widetilde{l}_{w_{1}}(A))$ 
is vanish. It is clear that
$(\widetilde{B}_{w_{1}})_{j}(\mathbf{v})$ is equal to  
$(\widetilde{B}_{w_{2}})_{j}(\mathbf{v})$ if $\mathbf{v}$ is not equal to
$\widetilde{l}_{w_{1}}(A)$.      
Therefore $(\widetilde{B}_{w})_{j}$ 
is invariant under the second homotopy move.\par 
Consider the third homotopy move
$$w_{1}:= (\mathcal{A},xAByACzBCt) \rightarrow 
w_{2}:=(\mathcal{A},xBAyCAzCBt)$$
where $\|A\|=\|B\|=\|C\|$, and $x$, $y$, $z$ and $t$ are words on 
$\mathcal{A}$ \par
Now we use same terminologies as in the proof of Proposition \ref{propSo}. 
We call a letter $A$ \emph{single component letter} if 
$A \in \mathcal{A}_{jj}$.
First we consider the case $A$, $B$ and $C$ are single component letters.
We show $\widetilde{l}_{w_{1}}(X)$ is equal to 
$\widetilde{l}_{w_{2}}(X)$ for all $X \in \mathcal{A}$.
For all $D \in \mathcal{A}$ $\setminus$ $\{A,B\}$,
it is clear that  $\widetilde{l}_{w_{1}}(D)$ 
is equal to $\widetilde{l}_{w_{2}}(D)$.\par
On $\widetilde{l}_{w_{1}}(A)$ is equal to 
$\widetilde{l}_{w_{2}}(A)$.
By the definition of $\widetilde{l}_{w_{1},j}$,
\begin{eqnarray*}
\widetilde{l}_{w_{1},j}(A) &=& \sum_{X \in \mathcal{A}} 
                                 \widetilde{\sigma}_{w_{1}}^{j}(A,X)\\
                         &=&  \sum_{X \in \mathcal{A} \setminus \{A,B\}} 
                                 \widetilde{\sigma}_{w_{1}}^{j}(A,X)
                               + \widetilde{\sigma}_{w_{1}}^{j}(A,B)
                               + \widetilde{\sigma}_{w_{1}}^{j}(A,C)
                               \cdots \cdots (\ast) 
\end{eqnarray*}
It is easily checked that 
$\widetilde{\sigma}_{w_{1}}^{j}(A,C)$ and
$\widetilde{\sigma}_{w_{2}}^{j}(A,B)$ are equal to $0$.
Since $\cdots$ $A$ $\cdots$ $B$ $\cdots$ $A$ $\cdots$ $B$ $\cdots$ in $w_{1}$,
$\cdots$ $A$ $\cdots$ $C$ $\cdots$ $A$ $\cdots$ $C$ $\cdots$ in $w_{2}$
and the second subscript of $\|A\|$ and 
the second subscript of $\|B\|$ are equal,
we obtain  $\widetilde{\sigma}_{w_{1}}^{j}(A,B)$ is equal to 
$\widetilde{\sigma}_{w_{2}}^{j}(A,C)$. 
Thus 
\begin{eqnarray*}
(\ast) &=&  \sum_{X \in \mathcal{A} \setminus \{A,B\}} 
                                 \widetilde{\sigma}_{w_{2}}^{j}(A,X)
                               + \widetilde{\sigma}_{w_{2}}^{j}(A,B)
                               + \widetilde{\sigma}_{w_{2}}^{j}(A,C)\\
       &=& \widetilde{l}_{w_{2},j}(A). 
\end{eqnarray*}
On $\widetilde{l}_{w_{1}}(B)$ is equal to 
$\widetilde{l}_{w_{2}}(B)$. By the definition of $\widetilde{l}_{w_{1},j}$,
\begin{eqnarray*}
\widetilde{l}_{w_{1},j}(A) &=& \sum_{X \in \mathcal{A}} 
                                 \widetilde{\sigma}_{w_{1}}^{j}(B,X)\\
                         &=&  \sum_{X \in \mathcal{A} \setminus \{B,C\}} 
                                 \widetilde{\sigma}_{w_{1}}^{j}(B,X)
                               + \widetilde{\sigma}_{w_{1}}^{j}(B,A)
                               + \widetilde{\sigma}_{w_{1}}^{j}(B,C)
                               \cdots \cdots (\ast\ast) 
\end{eqnarray*}
Then we show $\widetilde{\sigma}_{w_{1}}^{j}(B,A) +   
\widetilde{\sigma}_{w_{1}}^{j}(B,C) = 0$.
In fact, if $|A| \in \widehat{a_{i}}$ (in this case $|B|,|C| \in \widehat{a_{i}}$)
and $\widehat{a_{i}}$ is a fixed point of $\tau$,
then  $\widetilde{\sigma}_{w_{1}}^{j}(B,A)$ is equal to $\mathbf{e}_{(i,i)}$ and 
$\widetilde{\sigma}_{w_{1}}^{j}(B,C)$ is equal to $\mathbf{e}_{(i,i)}$ if 
$j$ is equal to $i$. Otherwise $\widetilde{\sigma}_{w_{1}}^{j}(B,A)$ 
is equal to $\mathbf{0}$ and 
$\widetilde{\sigma}_{w_{1}}^{j}(B,C)$ is equal to $\mathbf{0}$.
Since $K_{(i,i)}$ is equal to $\mathbb{Z}/2\mathbb{Z}$, 
we obtain  $\widetilde{\sigma}_{w_{1}}^{j}(B,A)$ $+$   
$\widetilde{\sigma}_{w_{1}}^{j}(B,C)$ $=$ $2\mathbf{e}_{(i,i)}$
$=$ $0$.
If  $|A| \in \widehat{a_{i}}$ and $\widehat{a_{i}}$ is a free orbit of $\tau$,
then we obtain $\widetilde{\sigma}_{w_{1}}^{j}(B,A)$ is equal to
$-\widetilde{\sigma}_{w_{1}}^{j}(B,C)$.
Thus $\widetilde{\sigma}_{w_{1}}^{j}(B,A)$ $+$   
$\widetilde{\sigma}_{w_{1}}^{j}(B,C)$ $=$ $0$.
Therefore we obtain 
$$(\ast\ast) =  \sum_{X \in \mathcal{A} \setminus \{B,C\}} 
                                 \widetilde{\sigma}_{w_{1}}^{j}(B,X)=
                 \widetilde{l}_{w_{2},j}(B).$$
On $\widetilde{l}_{w_{1}}(C)$ is equal to 
$\widetilde{l}_{w_{2}}(C)$.
By the definition of $\widetilde{l}_{w_{1},j}$,
\begin{eqnarray*}
\widetilde{l}_{w_{1},j}(C) &=& \sum_{X \in \mathcal{A}} 
                                 \widetilde{\sigma}_{w_{1}}^{j}(C,X)\\
                         &=&  \sum_{X \in \mathcal{A} \setminus \{A,B\}} 
                                 \widetilde{\sigma}_{w_{1}}^{j}(C,X)
                               + \widetilde{\sigma}_{w_{1}}^{j}(C,A)
                               + \widetilde{\sigma}_{w_{1}}^{j}(C,B)
                               \cdots \cdots (\ast\ast\ast) 
\end{eqnarray*}
It is easily checked that 
$\widetilde{\sigma}_{w_{1}}^{j}(C,A)$ and
$\widetilde{\sigma}_{w_{2}}^{j}(C,B)$ are equal to $0$.
Since $\cdots$ $B$ $\cdots$ $C$ $\cdots$ $B$ $\cdots$ $C$ $\cdots$ in $w_{1}$,
$\cdots$ $A$ $\cdots$ $C$ $\cdots$ $A$ $\cdots$ $C$ $\cdots$ in $w_{2}$
and the first subscript of $\|A\|$ and 
the first subscript of $\|B\|$ are equal,
we obtain  $\widetilde{\sigma}_{w_{1}}^{j}(C,A)$ is equal to 
$\widetilde{\sigma}_{w_{2}}^{j}(C,B)$. 
Thus 
\begin{eqnarray*}
(\ast\ast\ast) &=&  \sum_{X \in \mathcal{A} \setminus \{A,B\}} 
                                 \widetilde{\sigma}_{w_{2}}^{j}(C,X)\\
       &=& \widetilde{l}_{w_{2},j}(C). 
\end{eqnarray*} 
By the above, in this case we prove invariance of $\widetilde{S_{o}}$
by the third homotopy move.\par
On the case $A \in \mathcal{A}_{ii}$ and $B,C \not\in \mathcal{A}_{jj}$,
$A \in \mathcal{A}_{ii}$ and $A,B \not\in \mathcal{A}_{jj}$,
and  $A,B,C \not\in \mathcal{A}_{jj}$, we can prove invariance of 
$\widetilde{S_{o}}$ by using a part of the above case.\par
Now we completed the proof.
\end{proof}  
Note that in the proof of invariance of $\widetilde{S_{o}}$ by the 
first homotopy move, we did not use the condition $\|A\| \in Q_{\alpha,k}$.
Thus we obtain a following corollary.
\begin{corollary}
The $\widetilde{S_{o}}$ is an $S_{\alpha,k}$-homotopy invariant for 
nanowords over $\alpha_{k}$.
\end{corollary} 
\subsection{Final Remark}
In this section, we extended the invariants $lk$, $w$ and extended $S_{o}$
for nanophrases over $\alpha$ to the
$(Q_{\alpha,k},R_{\alpha,k},S_{\alpha,k})$-homotopy invariants 
for nanowords over $\alpha_{k}$.
As a similarly, we can extend the invariant $\gamma$ for nanophrase
over $\alpha$ with homotopy data $\Delta_{\alpha^{3}}$ which is 
defined in \cite{fu1}(see also \cite{tu1}) 
to a $(Q_{\alpha,k},R_{\alpha,k},$ $(\Delta_{\alpha^{3}})_{\alpha,k})$-homotopy
invariant for nanowords. Moreover we can extend the lemmas 5.2 and 5.3 
in the paper \cite{fu2}.
A natural question "Can we extend any $S$-homotopy invariant for nanophrases
over $\alpha$ to a $(Q_{\alpha,k},R_{\alpha,k},S_{\alpha,k})$-homotopy invariant
for nanowords over $\alpha_{k}$?" is a future problem.\\
\\ 
\begin{acknow}
The author would like to express to my gratitude to Goo Ishikawa
for many useful comments especially on Section \ref{ext}.
Furthermore, the author would like to very thanks to Andrew Gibson 
for many useful comments on invariants
in Section 5. 
Thanks to that, invariants $S_{o}$
and $\widetilde{S}_{o}$ in Section 5 were improved.
\end{acknow}

\vspace{0.3cm}
Department of Mathematics, Hokkaido University

Sapporo 060-0810, Japan

e-mail: fukunaga@math.sci.hokudai.ac.jp

\end{document}